\documentclass[11pt,twoside]{atmp}

\usepackage{amsmath,amssymb,latexsym,graphicx,epsfig,color,hyperref,supertabular,enumerate}

\usepackage[all]{xy}

\theoremstyle{plain}
\newtheorem{theorem}{Theorem}

\newtheorem{proposition}[theorem]{Proposition}
\newtheorem{lemma}[theorem]{Lemma}
\newtheorem{assumption}[theorem]{Assumption}

\theoremstyle{definition}
\newtheorem{definition}[theorem]{Definition}
\newtheorem{example}[theorem]{Example}

\newtheorem{definitiont}[theorem]{Definition-Theorem}
\theoremstyle{remark}
\newtheorem{remark}[theorem]{Remark}
\newtheorem{notation}[theorem]{Notation}

\numberwithin{theorem}{section}

\def\C{\mathbb{C}}
\def\Z{\mathbb{Z}}
\def\R{\mathbb{R}}

\def\P{\mathbb{P}}

\def\D{\Delta}
\def\co{\mathcal O}

\def\co{\mathcal{O}}
\def\SO{\operatorname{SO}}
\def\SL{\operatorname{SL}}

\def\lg{\langle}
\def\rg{\rangle}

\def\X{{X_{\Sigma}}}

\def\pu{\mathbb P^1}
\def\nbl{\noindent $\bullet$ }

\begin{document}

\title[W-models, toric hypersurfaces, symplectic cuts]
{Weierstrass models of elliptic toric $K3$ hypersurfaces and symplectic cuts}

\arxurl{arXiv:1201.0930v2}

\author[A. Grassi and V. Perduca]{Antonella Grassi and Vittorio Perduca}

\address{Department of Mathematics, University of Pennsylvania \\ David Rittenhouse Laboratory, 209 S 33rd Street \\ Philadelphia, PA 19104, USA\\MAP5 - Laboratory of Applied Mathematics \\ Paris Descartes University and CNRS \\ 45 Rue des Saints-P\`eres \\ 75006 Paris, France}  %lines should be separated with double backslashes: \\
\addressemail{grassi@math.upenn.edu, vittorio.perduca@parisdescartes.fr}

%\dedicatory{A Luciana Picco Botta, in memoriam.}

\begin{abstract}We study  elliptically fibered $K3$ surfaces, with sections, in toric Fano threefolds which satisfy certain combinatorial properties relevant to F-theory/Heterotic duality. We show that some of these conditions are equivalent to  the existence of an appropriate notion of a Weierstrass model adapted to the toric context. Moreover, we show that if in addition other conditions are satisfied, there exists a toric semistable degeneration of the elliptic $K3$ surface which  is compatible with the elliptic fibration and F-theory/Heterotic duality.
\end{abstract}

\maketitle

%\section{<title_for_first_section>}

%<insert_text_for first section>

%\cutpage %move this line so that the first page breaks at the appropriate place.

%\setcounter{page}{<insert page # for second page>}

\noindent

\section{Introduction}

 This paper  begins an investigation motivated by the physics of String Theory,
 in particular the ``F-theory and Heterotic duality"  which predicts unexpected relations between certain $n$ dimensional elliptically fibered Calabi-Yau varieties with section (the F-theory models) and certain principal  bundles over $n-1$ dimensional Calabi-Yau varieties
(the Heterotic models), \cite{MV},  \cite{FMW} and \cite{FMWJAG}. Recently this duality has also been used to construct realistic models in string theory, see for example the survey \cite{TaylorTASI}.

 The duality is conjectured  to be between $n$ dimensional Calabi Yau  varieties which are $K3$ fibered  and $n-1$ dimensional Calabi-Yau variety  elliptically fibered over the same base \cite{FMW} and \cite{FMWJAG}. When $n=2$, that is the duality is between F-theory  on an elliptically fibered $K3$ surface with section and Heterotic theory on an elliptic curve,   the elliptic curve and the bundles are obtained by a suitable semistable degeneration of the $K3$ surface; this semistable degeneration is the starting point of the  duality in higher dimension \cite{MV}.

It is neither straightforward nor easy to produce these pairs of dual varieties.
In the 90s, Candelas and collaborators, \cite{CandelasFont}, \cite{CandelasNeuchatel}, \cite{CandelasCongress}, \cite{CandelasPerevalovRajesh}, \cite{Rajesh} and also \cite{BCDG}  proposed a clever and quick algorithm to find the Heterotic Calabi-Yau duals $(Y,E)$ for {\textit{certain}}  Calabi-Yau manifolds $X$ which are hypersurfaces in toric
 Fano varieties $V$.  Berglund and Mayr \cite{BeMa}  proposed an explicit construction which assigns to each given toric variety of dimension $n$  a
toric variety of dimension $n-1$ together with a specific family of sheaves on it. This
 systematic construction produces many interesting examples of conjectured F-theory/Heterotic dual in the toric context.
 The essence of these algorithms is a sequence of {\textit{suitable}} projections from the toric fan of $V$. Candelas and collaborators, Berglund and Mayr check via other physics dualities that the constructions produces the
expected pair of dual varieties and the group, but the mathematical reasoning  behind this, that is the connection to the semistable degeneration, remains elusive.

This paper represents a first step towards the ultimate goal, mentioned in the paper of Hu, Liu and Yau \cite{HYL},  to present the F-theory/Heterotic
analogue of Batyrev's constructions of toric mirrors.
We consider the case of $K3$ surfaces/elliptic curves, which are the building block for higher dimensional F-theory/Heterotic duality.
When the duality is between  F-theory on an elliptically fibered $K3$ surface with section and Heterotic theory on an elliptic curve,
   Clingher and Morgan proved that certain regions of
 the moduli space of such Heterotic theories and their F-theory counterparts can be identified as dual \cite{ClingherMorgan}.
 The authors consider a partial compactification of the moduli spaces of the smooth elliptic $K3$ surfaces with section, by adding two divisors at infinity $D_1$ and $D_2$; the points of the boundary divisors  correspond
to semistable degenerations of $K3$ surfaces given by the union of two rational
elliptic surfaces glued together along an elliptic curve in two different ways. An elliptic curve $E$ (the double curve of the degeneration)  is then shown to be the Heterotic dual of the $K3$; the elliptic curve is endowed with a flat $G$-bundle.  The  Lie group $G$ is
 $ (E_8 \times E_8) $ for one boundary divisor and
 $\mbox{Spin}(32)/ \Z/2\Z$ for the other.
 We give a mathematical definition for the properties of the examples of Candelas and collaborators, which we refer to as  Candelas' conditions in Section \ref{sec:candelasexamples} and proof several theorems, in particular \ref{cf} and \ref{toricbat}, which provide an explanation, in term of semistable degeneration, of Candelas' algorithm, and part of Berglund-Mayr's construction.
The main idea  is to combine techniques from symplectic geometry,  namely the symplectic cut of the moment polytope \cite{Lerman}, \cite{Ionel}, with an appropriate notion of a Weierstrass model, adapted to the toric context.
The dual Newton (moment) polytope is the polytope naturally associated to Fano toric varieties, whose lattice points correspond to the sections of the anticanonical bundle.

In Section \ref{background}, after some general facts about toric Fano varieties and their Calabi-Yau hypersurfaces  we explain how the inclusion of a $2$-dimensional reflexive polytope determines the elliptic fibration. This observation, long known in the physics literature, is the starting point of Rohsiepe's analysis for higher dimensional Calabi-Yau elliptic fibrations \cite{Ro05}: we use it to construct the elliptic fibrations in  the examples. Here ``top and bottom" are used to identify the inverse image(s) of the fixed points of $\P^1$  under the torus action, and the corresponding singular fibers of the elliptic fibration.

We give a summary of Candelas' conditions in Section \ref{sec:candelasexamples}. The surjection condition 2) appears in all the Candelas' examples: it is  a novel idea, somewhat hidden in Candelas' papers and central in Berglund-Mayr.  We recast the conditions in a mathematical context in the following Section \ref{polytopesfibrationssection}. In \ref{semistablepolys} we define the semistable polytopes, which play the role of  building blocks in the  examples in above mentioned papers.
A semistable polytope will satisfy the Candelas' conditions 1) and 3) stated in Section \ref{sec:candelasexamples}. The particular class of polytopes Theorem \ref{th:exun} satisfies
all the Candelas' conditions; in Section \ref{subsec:cf} we actually show that all the polytopes which satisfy the Candelas' conditions can be constructed from semistable polytopes.
 In section \ref{sections} we
show that condition 3) corresponds to the existence of a ``section at $\infty$" of the elliptic fibration of the K3 surface. We show that the existence of the section at $\infty $ and condition 1) implies the existence of a \emph{toric} Weierstrass model, which we call the \emph{Candelas-Font Weierstrass} model (Section \ref{subsec:cf}). We characterize these models in Theorem \ref{th:diagr}.
 Theorem \ref{th:projection} characterizes Candelas' examples.
It was noted by \cite{CandelasFont} and \cite{BeMa} that these Calabi-Yau varieties yield mirror elliptic fibrations.

In Section \ref{sec:tntsections} we proof an easy combinatorial condition for the existence of a section of the elliptic fibration.

The symplectic cut is the focus of Section \ref{symplecticcut}.  We show that Candelas'  conditions  1) and 2)  correspond to the existence of a codimension one slice in the moment polytope, which cuts it into two ``nice'' parts (Theorem \ref{semistablepartition}). We then show that these conditions give a ``symplectic cut'' which determines a toric  semistable degeneration of the original Fano variety into two other semistable toric varieties; this degeneration induces a natural semistable degeneration of the Calabi-Yau hypersurface,   (Theorems \ref{cf} and \ref{toricbat}). Remark \ref{rmk:cutswithinftysection} concludes  that Candelas' conditions 1), 2) and 3), or equivalently Theorem \ref{th:projection}, imply that  the elliptically fibered $K3$ degenerates to two rational elliptic surfaces glued along a fiber and the degeneration preserves the elliptic fibration which induces a semistable degeneration of the section at infinity.

%The dual Heterotic variety is then realized without taking any limit.
The idea of the symplectic cut can be applied also in higher dimensional fibrations; we leave this for further studies: we believe that \cite{HYL} contains many useful techniques for this purpose.

\noindent{\bf Acknowledgements.} We would like to thank V. Braun, P. Candelas, M. Rossi,  U. Whitcher  and in particular X. De la Ossa for several useful conversations. This paper is based in part on some sections of V. Perduca's thesis, at University of Turin. We would like to thank the  the Mathematics Department of University of Pennsylvania and il Dipartimento di Matematica dell' Universit\`a di Torino, especially A. Conte and M. Marchisio. A.G. would  like to thank J. Morgan and the staff at Simons Center for Geometry and Physics for their gracious hospitality.

This research project was partially supported by National Science Foundation grant DMS-0636606 and by the Simons Center for Geometry and Physics. V.P. benefited from INDAM - GNSAGA travel fellowships and is presently supported by the  \emph{Fondation Sciences Math\'ematiques de Paris} postdoctoral fellowship program, 2011-2013.

\section{Background and notations}\label{background}

\subsection{Toric, Fano varieties}%~\vskip 0.1in

We follow the notation of \cite{Bat94}, \cite{CoKa99} and \cite{CoxLSh}.

\nbl $N,M \subset \Z^n$ are dual lattices with real extensions $N_{\R},M_{\R}$; we denote by $\lg*,*\rg:M\times N\rightarrow\Z$ the natural pairing; $T_N =N\otimes\C^*$ is the algebraic torus;

\nbl $\D\subset M_{\R}$ is an integral  polytope, that is, each vertex is in $M$; the codimension $1$ faces of $\D$ are called facets;

\nbl We assume $\D$ is \emph{reflexive}, that is, the equation of any  facet $F$ of $\D$ can be written as $\lg m,v\rg=-1$, where $v\in N$ is a fixed integer point and $m\in F$; then the origin is the only  integral interior point  in  $\D$.
The dual of $\D \subset N_{\R}$, defined as the set $\nabla \stackrel{def}= \{v\in N_{\R}| \lg m,v\rg\geq -1 \mbox{ for all } m\in\D\}$, is also an integral reflexive polytope in $N_{\R}$.

\nbl The normal fan of $\D\subset M_{\R}$ in $N$ is the fan over the proper facets of $\nabla \subset N_{\R}$; since $\D$ is reflexive  the rays of its normal fan  are simply the vertices of $\nabla$; let $\P_{\D}$ be  the associated projective toric variety.

\nbl Given a fan $\Sigma$ in $N$, we denote as $X_{\Sigma}$ the corresponding toric variety; when the meaning is clear we simply write $X$.

\nbl $\Sigma^{(1)}$ is the set of all rays of $\Sigma$; each ray $v_i\in\Sigma^{(1)}$ corresponds to an irreducible $T_N$-invariant Weil divisor $D_i\subset X_{\Sigma}$, the \textit{toric} divisors.

\nbl   $\D$ is \emph{reflexive} if and only if the projective toric variety  $\P_{\D}$ is Fano.  Recall that the dualizing sheaf on a compact toric variety $X$ of dimension $n$ is $\hat{\Omega}^n_V=\co_{X}(-\sum_i D_i)$, where the sum ranges over all the toric divisors $D_i$. The canonical divisor is $K_X=-\sum_i D_i$, and therefore $\P_{\D}$ is \emph{Fano}, if and only if $\sum_i D_i$ is ample.

\nbl  A \emph{projective subdivision} $\Sigma$ is a refinement of the normal fan of $\Delta$ which is projective and simplicial,  that is, the generators of each cone of $\Sigma$ span $N_\R$.  The associated toric variety $X_{\Sigma}$ has then orbifold singularities.  $\Sigma$ is \emph{maximal} if its cones are generated by all the lattice points of the facets of $\Sigma$.

 \begin{definitiont}[The Cox ring \cite{Co95}] For each $v_i\in \Sigma^{(1)}$ introduce a variable $x_i$ and consider the polynomial ring
$$
S=\C[x_i:\,v_i\in\Sigma^{(1)}]=\C[x_1,\ldots,x_r],
$$
where $r=|\Sigma^{(1)}|$. $S$ is graded by $A_{n-1}(X_{\Sigma})$  and is called the \emph{homogeneous (Cox) coordinate ring}
 of $X_{\Sigma}$.
 A monomial $\prod_i x^{a_i}\in S$ has degree $[D]\in A_{n-1}(X_{\Sigma})$, where $D=\sum_ia_iD_i$.
 \end{definitiont}

 \begin{definitiont}
For each cone $\sigma\subset \Sigma$ consider the monomial $x^{\hat{\sigma}}=\prod_{v_i\notin \sigma}x_i\in S$, and define the \emph{exceptional set} associated to $\Sigma$ as the algebraic set in $\C^r$ defined by the vanishing of all of these monomials:
$$
Z(\Sigma)=V(x^{\hat{\sigma}}:\, \sigma\in\Sigma)\subset \C^r.
$$
Finally, define $$
G=\{(\mu_1,\ldots,\mu_r)\in (\C^*)^r|\prod_{i=1}^r\mu_i^{\lg e_1,v_i\rg}=\ldots=\prod_{i=1}^r\mu_i^{\lg e_n,v_i\rg}=1\}\subset (\C^*)^r,
$$
where $\{e_1,\ldots,e_n\}$ is the standard basis in $M$. Then: $$X_{\Sigma} \simeq (\C^r-Z(\Sigma))/G.$$ \end{definitiont}

\subsection{Calabi-Yau varieties and reflexive polytopes}%%~\vskip 0.1in

$V$ is a \underline{Calabi-Yau variety} if $K_V \sim \co (V)$, $h^i(\co _V)=0, \ 0<i<\dim V$. If $V$ is an hypersurface in a toric variety, then the condition $h^i(\co _V)=0$ is automatically satisfied.

In fact the lattice polytope $\D_D$ corresponds to the very ample Cartier divisor $D$ which determines the embedding of $X_{\Sigma}$ in some projective space; $\Sigma$ is the normal fan to $\D_D$ . On the other hand, let $\P_{\D}:=X_{\Sigma}$ be the toric variety associated to the normal fan $\Sigma$ of $\D$. $\D$ determines a Cartier divisor $D_{\D}$ and an ample line bundle $\mathcal{L}_{\D}$; its global sections (which provide the equations of the projective embedding) corresponds to the lattice points of $\nu\D$ as explained above.

\begin{theorem}[Ch. 4 \cite{CoKa99}]
If $\D\subset M_{\R}\simeq\R^n$ is a reflexive polytope of dimension $n$, then the general member $\bar{V}\in|-K_{\P_{\D}}|$ is a Calabi-Yau variety of dimension $n-1$. If $\Sigma$ is a projective subdivision of the normal fan of $\D$, then
\begin{itemize}
\item $X_{\Sigma}$ is a Gorenstein orbifold with at worst canonical singularities;
\item $-K_{X_{\Sigma}}$ is semiample and $\D$ is the polytope associated to $-K_{X_{\Sigma}}$;
\item the general member $V\in|-K_{X_{\Sigma}}|$ is a Calabi-Yau orbifold with at worst canonical singularities.
\end{itemize}
\end{theorem}

In particular, in dimension three or lower the following are equivalent:
\begin{enumerate}
\item $\Sigma=\Sigma_{\max}$ is maximal;
\item $\Sigma$ is given by a triangulation of the facets of $\nabla$ into elementary triangles $v_{i_1}v_{i_2}v_{i_3}$ such that for all $i$, the vectors $v_{i_1},v_{i_2},v_{i_3}$ span the lattice $N$ (equivalently, the convex hull of $\{v_{i1},v_{i2},v_{i3},\mathbf{0}\}$ is a tetrahedron with no lattice points other than its vertices).
 \item $ X_{\Sigma_{\max}}$    is smooth
\end{enumerate}

If $n=3$, $V_{\max}$ is a smooth $K3$ surface, while if $n=2$, $V_{\max}$ is an elliptic curve.

\begin{remark}\label{rem:cyeq} The defining equation of the Calabi-Yau hypersurface $V \in|-K_{X_{\Sigma}}|$, with $X_{\Sigma} $  toric Fano, can be written explicitly. With the above notation, let $z_1,\ldots,z_k$ be the  homogeneous coordinates of $X_{\Sigma}$. $V$ is defined by the vanishing of the generic polynomial whose monomials are the sections of the line bundle $\co_{X_{\Sigma}}(-K_{X_{\Sigma}})$; then $V$ has equation
\begin{equation*}
\label{eq:gency}
\sum_{m\in\D\cap M}a_m\prod_{i=1}^kz_1^{\lg m, v_1\rg+1}\cdot z_2^{\lg m,v_2\rg+1}\cdot\ldots\cdot z_k^{\lg m,v_k\rg+1}=0,
\end{equation*}
where the $a_m$s are generic complex coefficients. The defining equation of $V$ is invariant modulo the action of $\SL(3,\Z)$ on $M$.   These equations are easily implemented in the computer algebra system SAGE \cite{sage} and \cite{Perduca} which has a dedicated package for working with reflexive polytopes.
\end{remark}

 Batyrev \cite{batyrev1984higher}, Koelman \cite{koelman1991number} and then \cite{KrSk97} independently classified all the reflexive polytopes of dimension 2 up to $\SL(2,\Z)$ transformations, see  Figure~\ref{fig:sections2d}.  In dimension 3, which is the one relevant for $K3$ surfaces, the complete classification was carried out by Kreuzer and Skarke \cite{KrSk97,KrSk98} using the software package PALP \cite{KrSk04}. There are 4319 reflexive polytopes in dimension three. Their coordinates are stored in SAGE and can be found on the web page \cite{KrSkweb}.
See also \cite{KnappKreuzer,KnappKMW}.
 Rohsiepe's tables \cite{Ro205} contain a list of all three dimensional reflexive polytopes, in the same ordering as in SAGE, the only difference being that the first reflexive polytope in Rohsiepe's tables is indexed by 1 whereas SAGE indexing starts from 0. Throughout this paper we adopt the same indexing as in SAGE.

\subsection{Intersection on toric $K3$ hypersurfaces}
\label{subsec:intersection}

The following facts about the intersection on toric $K3$ hypersurfaces are known in physics literature \cite{PeSk97}.

Let $\Sigma$ be a projective subdivision of the fan over the proper facets of a  3-dimensional reflexive polytope $\nabla$. For each ray $v_i$ in $\Sigma$, let $D'_i$ be the intersection of $D_i\subset X_{\Sigma}$ with the general $K3$ hypersurface $V\in |K_{X_{\Sigma}}|$. Three cases can occur: 1) $D_i$ doesn't intersect $V$, i.e. $D_i'=0$; 2) $D_i'$ is irreducible on $V$: we call it a \textit{toric} divisor; 3) $D_i'$ is the sum of irreducible divisors on $V$: we call its irreducible components \textit{non toric} divisors.

Let $\Sigma$ be \textbf{maximal}. In this case: 1) $D'_i=0$ if $v_i$ is in the interior of a facet of $\nabla$; 2) $D'_i$ is toric if $v_i$ is a vertex of $\nabla$; 3) $D'_i$  is the sum of $l'(\theta^*) + 1$ non toric divisor if $v_i$ is in the interior of an edge $\theta$ of $\nabla$, where $l'(\theta^*)$ is the \emph{lattice length} of the dual of $\theta$ (i.e. the number of lattice points in the interior of $\theta^*\subset\D$) \cite{CoKa99}. Now let $v_1, v_2$ be two distinct rays in $\Sigma$. The intersection $D_1\cdot D_2\cdot V$ can be non- zero iff $v_1,v_2$ are in the same cone in $\Sigma$, that is there are two elementary triangles $T,T'$ in the triangulation of the facets of $\nabla$ that have the segment $v_1v_2$ in common. Let $v_3$ be the third vertex of $T$ and $v_4$ be the third vertex of $T'$, and denote as $m_{123}\in M$ the dual of the facet of $\nabla$ carrying $T$. Then

\begin{theorem}[\cite{PeSk97}]
\label{th:intK31}
$D_1\cdot D_2\cdot V =\lg m_{123},v_4\rg +1$. In particular, $D_1\cdot D_2\cdot V=0$ if $v_1$ and $v_2$ are not neighbors along an edge of $\nabla$. If $v_1$ and $v_2$ are neighbors along an edge $\theta_{12}$, then
$$
D_1\cdot D_2\cdot V = l_{12} = l'(\theta^*_{12})+1,
$$
where $l'(\theta^*_{12})$ is the lattice length of $\theta^*_{12}$.
\end{theorem}

\bigskip

\subsection{Elliptic Fibrations}%~\vskip 0.1in

The morphism $\pi_V: V \to B$ denotes an \emph{elliptic fibration}, that is
 $\pi_V^{-1} (p)$ is a smooth elliptic curve   $ \forall \ \ p \ \in B$, general. In addition,  $\pi_V: V \to B$ is an \emph{elliptic fibration with section} if there exists a morphism $\sigma_V: B \to V$  which composed with $\pi_V$ is the identity; $ \sigma_V (B)$ is a section of $\pi$.

$\phi: \X \to B$ is a \emph{toric} fibration if $X_{\Sigma}$ (from here on denoted simply as $X$) and $B$ are toric and if $\phi$ is induced by a lattice morphism between the corresponding lattices $\varphi: N_X \to N_B$ which is compatible with the fans. The kernel of $\varphi$ is a sublattice $N_{\varphi} \subset N$.

   We now assume that $X$ is a Fano variety and that $X_{\phi}$ the  general fiber of $\phi$   is  a Fano surface; the restriction of  the fan of $X$ to $N_{\varphi}$ defines the fan of $X_{\phi} $.
 This fibration determines a 2-dimensional   reflexive polytope $\nabla^{\varphi} \subset N_{\varphi,\R}$ corresponding to the toric
 variety $X _\phi$; let $E \in|-K_{X_\phi}|$ be a  general element, a smooth elliptic curve.

\begin{assumption}  We also assume that there  is a section $\sigma: B \to V$ of $\pi: V \to B$
 {induced by a \textbf{toric section}} $\sigma_{X}: B \to X$ of $\phi: X \to B$  such that $\sigma_{X}(B)=D$ is
 a toric divisor. The restriction of $D$ to $V$ can be either reducible or irreducible.
 \end{assumption}

Nakayama showed that an elliptic fibration $V\rightarrow B$ with section has a Weierstrass model in a precisely defined projective bundle \cite{Na88}. In particular, when $V$ is a $K3$ surface and $B=\P^1$, the projective bundle is
$\textbf{P}=\P(\co_ {\P^1} \oplus \co_ {\P^1}(-4) \oplus \co_ {\P^1}(-6))$. Every projective bundle is a toric variety whose corresponding fan can be computed by following the construction described by Oda, Section 1.7 \cite{Od88}. However it turns out that the toric variety $\mathbf{P}$ is not Fano \cite{CoKa99}. For the elliptic $K3$s which are hypersurfaces in Fano toric threefolds, we would like their Weierstrass models to be hypersurfaces in a Fano toric threefolds as well.

\smallskip

We focus on 3-dimensional toric varieties and elliptic K3  hypersurfaces.  As noted in \cite{KrSk98} and \cite{Ro05}, if a reflexive 3-dimensional polytope $\nabla$ contains a reflexive subpolytope $\nabla^\varphi$, then a suitable refinement of the normal fan   of $\Delta$ always gives rise to a natural toric fibration $X \to \pu$ with general fiber a toric Fano surface with reflexive polytope $\nabla^\varphi$.
This can be seen explicitly by describing the toric morphisms in homogeneous coordinates, \cite{Perduca} and the Appendix of the present paper. Other explicit examples of elliptic fibrations of toric $K3$s in homogeneous coordinates can be found in \cite{AvKrMaSk97}. Rohsiepe searched all 4319 reflexive polytopes for toric elliptic fibrations \cite{Ro05}, the results can be found in the tables \cite{Ro205}.

\section{Candelas' Examples}\label{sec:candelasexamples}
We summarize here the main characteristic of the Candelas' examples of elliptic $K3$ F-theory models: recall that $Y \subset V$ is an anticanonical general surface in the toric Fano threefold $X$; all the statements are up to a lattice automorphism $\SL(3,\Z)$.
\begin{enumerate}
\item The lattice $N_\varphi \subset N$ is a summand of $N$, with induced morphism of lattices: $ N \twoheadrightarrow N_\varphi$  $(z_1, z_2, z_3) \mapsto (z_1,  z_{2})$.
\item Under this morphism the lattice  points of the reflexive polytope $\nabla$ are sent onto points of the reflexive polytope $\nabla ^ \varphi$.
\item There exists a vertex $v_z=(a,b,0)$ of $\nabla ^\varphi$ which is not a vertex of $\nabla$.
\end{enumerate}
 Note that conditions 2) and 3)  imply that the edge $L$ of $ \nabla^\varphi$ through $v_z$  is defined by the equation $z_1=a, \ z_2=b$.
 Condition 1) induces a  split of the dual lattice $M$, with coordinates $(z^*_1, z^*_2, z^*_3) $; $\Delta \cup \{z_3 ^*=0\}$, $\Delta ^\varphi$, the dual of the polytope $\nabla^\varphi$.

In  Section \ref{polytopesfibrationssection} we show that condition 3) corresponds to the existence of a ``section at $\infty$" of the elliptic fibration; in Section \ref{symplecticcut} we also assume that  $X$ is simplicial, which is also a case in the Candelas' examples, and show that conditions 1) and 2) imply the existence of the semistable degeneration of $X$ and $Y$. We then discuss the case when all the conditions are satisfied.

\begin{example}[Polytope 3737]\label{3737}  Let $\nabla$ be the polytope with vertices $v_s=(-1,-1,1)$, $v_t=(-1,-1,-1)$, $v_a=(-2,1,1)$, $v_b=(-2,1,-1)$, $v_c=(-1,1,1)$, $v_d=(-1,1,-1)$. The dual $\D$ is the \emph{diamond} in the Example~\ref{ex:diamond15}; $\nabla^{\varphi}$ is the 2-dimensional reflexive polytope number 15 given by points $v_x=(2,-1,0)$, $v_y=(-1,1,0)$, $v_z=(-1,-1,0)$. Clearly all the conditions 1), 2), 3) hold.

\end{example}

\begin{example}[Polytope 4318 fibered by 9]\label{4318}
Let $\nabla$ be the polytope with vertices $v_x=(-1,1,0)$, $v_t=(2,-1,0)$, $v_p=(-1,-1,-6)$, $v_q=(-1,-1,6)$, this is the polytope 4318 in the list \cite{KrSkweb}. We consider the fiber  $\nabla^{\varphi}$ given by points $v_x$, $v_y=(1,-1,-2)$, $v_z=(1,-1,2)$; $\nabla^{\varphi}$ is the 2-dimensional polytope number 9.  Consider the projective subdivision of the normal fan to $\D$ given by the rays $v_x,\,v_y,\,v_z,\,v_s = (0,-1,4),\,v_t,\,v_p,\,v_q$. $\nabla$ satisfies the conditions 1) and 3) but not 2). Then this example cannot have a semistable degeneration as described in Section \ref{symplecticcut} and an Heterotic dual with gauge group $E_8
\times E_8$. In fact  a singular fiber of the elliptic fibration is of Kodaira type $I^*_{12}$, which correspond to a ``gauge group"  of type $\SO(32)$ (or better, $\mbox{Spin}(32/\Z/2\Z)$). We will consider again this polytope in the Example~\ref{89} (with a different fibration).

%The monomials defining the generic $K3$ in the corresponding toric variety are \todo{TOGLIERE?}
%$$ y^2z^2st^3,\, y^2z^2p^6q^6s^4,\, y^2z^2p^4q^4s^3t,\, y^2z^2p^2q^2s^2t^2,\,x^2,\,z^4q^{12}s^8,\,y^4p^{12},\,xyzp^3q^3s^2,\,xyzpqst.$$

%$\nabla\supset\nabla_{9}^s$, where $\nabla_9^s$ has vertices $v_x,v_y,v_s,v_t$; the equation of the Candelas Weierstrass model $W\subset X_9^s$ have monomials
%$$
%a_8z^4,\,x^2,\,y^4,\,a_4y^2z^2,\,a_2xyz
%$$
%where $a_i$ is a general polynomial in $s, t$ of degree $i$, for $i=8,4,2$. The map $X\rightarrow X_{9}^s$ s.t. $[x,y,z,s,t]\mapsto [x,y,z,s,t,p=1,q=1]$ is a birational morphism between $V$ and the $K3$ in $X_{9}^s$ obtained by setting to zero the coefficients of the monomials of $W$ corresponding to the points of $\D$ which are not in $\D_{9}^s$ \todo{OK? Aggiungere qualcosa (fatto che la condizione di projezione non \'e rispettata - non so come dirlo)?Giuste notazioni?}
\end{example}

\begin{example}[Polytope 113: ``Diamond" fibered by 15] \label{ex:diamond15}%with vertex and coordinates does not have a section at infinity.
Let $\nabla$ be the reflexive polytope with vertices $v_x=(2,-1,0)$, $v_y=(-1,1,0)$, $v_z=(-1,-1,0)$, $v_s=(0,0,1)$, $v_t=(0,0,-1)$; $\nabla^{\varphi}$ has vertices $v_x,v_y,v_z$ and is the 2-dimensional polytope number 15. Conditions 1) and 2) are fulfilled while condition 3) does not hold.
%The Fano surface $X_{\nabla^{15}}$ \todo{(Notazione non va bene...)}  is the general fiber of the toric fibration.
%The equation of the general $K3$ hypersurface $V$ in the associated Fano toric variety is:
%$$\phi_0x^3+
%\phi_1xyz+
%\phi_2z^6+
%\phi_3y^2+
%\phi_4xz^4+
%\phi_5x^2z^2+
%\phi_6yz^3=0
%$$
%with each $\phi_j(s,t)$ a generic polynomial of degree $2$ in $(s,t)$.
\end{example}

\begin{example}[Polyotope 4: ``Diamond'' fibered by 1]\label{example:diamond1}
 Let $\nabla$ be the reflexive polytope with vertices $v_x=(1,0,0), v_y=(0,1,0), v_z=(-1,-1,0), v_s=(0,0,1), v_t=(0,0,-1)$; $\nabla^{\varphi}$ has vertices $v_x,v_y,v_z$ and is the 2-dimensional polytope number 1, the fan of $\P^2$. Conditions 1) and 2) are fulfilled and condition 3) does not hold.  The general $K3$ (in the anticanonical system) does not have a section,  this
is in fact the hypersurface in $\P^2\times \P^1$ of degree $(3,2)$.
%The Fano surface $X_{\nabla^{15}}$ \todo{(Notazione non va bene...)}  is the general fiber of the toric fibration.
%The equation of the general $K3$ hypersurface $V$ in the associated Fano toric variety is:
%$$\phi_0 x^3+\phi_1y^3+
%\phi_2 z^3+
%\phi_3 xy^2+
%\phi_4 x^2y+
%\phi_5 x^2z+
%\phi_6 yz^2+
%\phi_7 y^2z+
%\phi_8 yz^2+
%\phi_9 xyx=0
%$$ with each $\phi_j(s,t)$ a generic polynomial of degree $2$ in $(s,t)$.
\end{example}

\section{Semistable polytopes,  Sections at infinity, Weierstass models (\emph{Candelas' conditions recasted}).}\label{polytopesfibrationssection}

We start by defining the semistable polytopes, which play the role of  building blocks in the Candelas' examples:
a semistable polytope will satisfy the Candelas' conditions 1) and 3) stated in Section \ref{sec:candelasexamples}. The particular class of polytopes in Theorem \ref{th:exun} satisfies
all the Candelas' conditions; in Section \ref{subsec:cf} we actually show that all the polytopes which satisfy the Candelas' conditions can be constructed from semistable polytopes.

\subsection{Semistable polytopes}\label{semistablepolys}(\emph{On conditions 1) and 3)})%~\vskip 0.1in

Two-dimensional reflexive polytopes $\nabla^{\varphi}$ are classified up to $\SL(2, \Z)$ (see for example \cite{CoxLSh}), and are listed in Figure~\ref{fig:sections2d}. We denote by
 $\nabla^{i; \varphi}$ the i-th 2-dimensional reflexive polytope in the figure, and by $\D^{d(i);\varphi}$ its dual.
 %$\Sigma_{i,\varphi}$ is the normal fan to $\D_{i,\varphi}$, the fan over $\nabla^{i, \varphi}$.
Then
$d(1)=6,d(6)=1,d(2)=7,d(7)=2,d(3)=8,d(8)=3,d(4)=9,d(9)=4,d(5)=10,d(10)=5,d(11)=16,d(16)=11,d(12)=12,d(13)=13,d(14)=14,d(15)=15$.
\begin{definition} Fix a vertex $v_z$ of a 2-dimensional reflexive polytope $\nabla^{ \varphi} \subset N_{\varphi,\R} \subset N_{\R}$ such  that $N_{\varphi} \subset N$. Denote by $v_z$  both the point in $N$ and the unit vector on  a ray $\Sigma_{\varphi}$; let $L$ be a segment of lattice length $2$ centered at $v_z$ such  that the vertices $v_s$ and $v_t$ of $L$ together with the vertices of $\nabla ^\varphi$ generate   $N_\R$.
Assume that the lattice polytope $\nabla^s \subset N_\R$, spanned by the vertices of $\nabla^{\varphi}$ and  $L$ is reflexive: $\nabla^s$ is  a \emph{\textbf{semistable polytope}} with fiber $\nabla^{ \varphi}$.
\end{definition}

\begin{notation}\label{notationss} We denote a semistable polytope with $\nabla^{i, v_z, L; s}$ and its dual with $\D^{i, v_z, L;s}$; here $s$ stands for \emph{semistable}. The fan over $\nabla ^{i, v_z, L;s}$ has simplicial cones over the vertices of $\nabla^\varphi$, $v_s$ and $v_t$ and it determines a simplicial Fano toric variety $X^{i, v_z, L;s}$ together with a morphism $X ^{i, v_z, L;s}\to \pu$. For ease of notation we drop one or more apices among $i,v_z,L$ when the meaning is clear from the context.
\end{notation}
The general hypersurface $V \subset X^s$ in $|-K_{X^s}|$ is  an elliptically fibered $K3$ surface.

\begin{theorem}\label{th:exun}(Existence and uniqueness)
Assume that the vertices $\{ v_s, v_t \}$ together with the vertices of $\nabla ^{\varphi}$ generate  the lattice $N$. Then the convex hull of these vertices is a reflexive polytope and is a semistable polytope $\nabla^s$.  Moreover the fiber of the elliptic fibration  of a very general \footnote{A property is said to be  {\em very general} if it holds in the complement of a countable union of subschemes of positive codimension \cite{Lazarsfeld}.} $K3$ surface $V_{\max} \to \pu$ over the fixed points of $\pu$ are smooth, while the  other singular fibers are semistable.
This polytope is unique up to $\SL(3,\Z)$.
\end{theorem}
\begin{proof}   Let $v_x, v_y, v_z,v_{w_j}$ be the vertices of $\nabla^{\varphi}$ (depending on $\nabla^{\varphi}$ either there is no $v_{w_j}$ or $j\in\{1,2,3\}$); without loss of generality we take  $v_z=(z_1,z_2,0)$ and $v_s=(z_1,z_2,1)$ and $ v_t=(z_1,z_2,-1)$.
Let us consider the fan $\Sigma_{\max}$ of the maximal resolution of $\nabla^s$; by construction the rays of this fan are either $v_s, v_t$ or are also rays of  $\nabla^\varphi$. Note that the semistable polytope is simplicial and then the very general $K3$ surface  $V_{\max} \subset X_{\max}$ has the same Picard number of the  ambient Fano threefold $X_{\max} $  \cite{BG}. Then the singular fibers of the elliptic fibration are either nodes or restrictions of toric divisors corresponding to points of $\nabla^{\varphi}$ which are reducible when restricted to the $K3$. From  Section \ref{subsec:intersection} we see  that for each edge $e$ in $\nabla^{\varphi}$ which is also an edge of $\nabla$  there are $r$ semistable fibers $I_{q+1}$ in $V_S$, where $q$ and $r$ are the lattice lengths of $e$ and its dual in $\nabla$ respectively. This implies the first statement.

  For each edge $e$ of $\nabla^{\varphi}$ let $\langle m_e,v \rangle+1=0$ be its equation; $m\in\Z^2$ because $\nabla^{\varphi}$ is reflexive. If $e$ is one of the two edges originating from $v_z$, then the equation of the vertical facet defined by $e$ and $L$ is $\langle m,v\rangle+1=0$, where $m=(m_e,0)\in\Z^3$. If $e$ doesn't pass through $v_z$, the equation of the facet through $e$ and $v_s$ is $\langle m,v \rangle+1=0$ where $m=(m_e,-(az_1+bz_2+1))\in\Z^3$ (and similarly for the facets through $e$ and $v_t$). It follows that $\nabla$ is reflexive.
\end{proof}

The condition that $\{L,\nabla^{\varphi}\}$ generates $N$ does not hold for all the semistable polytopes, see Example \ref{1943}. However, this condition is fulfilled when $L$ is centered in a vertex of $\nabla^{\varphi}$ which satisfies a nice combinatorial property:
 %see Proposition~\ref{prop:rif_sum_radd}:

\begin{proposition}
\label{prop:rif_sum_radd}
Let $\nabla^s\subset N_\R$ be a semistable polytope and $L$ its edge of lattice length 2 centered in a vertex $v_z$ of $\nabla^\varphi\subset N_{\varphi,\R}$. Let $v_1,v_2$ be the two lattice neighbors of $v_z$ along the two edges of $\nabla^\varphi$ through $v_z$. If $v_z=v_1+v_2$ then $\{L,\nabla^\varphi\}$ generates $N$.
\end{proposition}
 The condition $v_z=v_1+v_2$ is not always satisfied: the vertices fulfilling this condition are marked with a square in Figure~\ref{fig:sections2d}; see also Section \ref{sec:tntsections}
Compare the above Proposition \ref{prop:rif_sum_radd} with Theorem \ref{prop:criterio_sections}.
\begin{proof}
Suppose $v_1=(a_1,a_2,0),\, v_2=(b_1,b_2,0)$ and let $v_s=(\alpha,\beta,\gamma)$ be the vertex of $L$ with $\gamma>0$. We show $\gamma=1$ by proving that if $\gamma>1$ then there is a lattice point in the interior of the segment $v_zv_s$.
%A similar argument shows that the other vertex of $L$ has coordinates $v_t=(\alpha',\beta',-1)$.

It can be easily checked that for the vertices $v_z$ s.t. $v_z=v_1+v_2$ (denoted by a square in Figure~\ref{fig:sections2d}) we always have $d:= a_1b_2-a_2b_1=\pm 1$. %\todo{(C'e' un qualche motivo geometrico preciso per questo fatto?)}.
The facet of $\nabla^s$ through $v_z,v_s,v_1$ has equation $\langle m_{zs1},v \rangle+1=0$, where
\begin{equation*}
\label{eq:sz1}
m_{zs1} = \left(-b_2d^{-1},\, b_1d^{-1},\, \frac{-b_1\beta d^{-1}+b_2\alpha d^{-1} - 1}{\gamma}\right)\in M,
\end{equation*}
Similarly, the facet through $v_z,v_s,v_2$ has equation $\langle m_{zs2},v \rangle+1=0$, where
\begin{equation*}
\label{eq:sz2}
m_{zs2} = \left(a_2d^{-1}, -a_1d^{-1}, \frac{a_1\beta d^{-1}-a_2\alpha d^{-1}-1}{\gamma}\right)\in M.
\end{equation*}
It follows that
%$$
%\left\{
%\begin{array}{cl}
%-b_1\beta d^{-1}+b_2\alpha d^{-1} - 1&\equiv0\mbox{ mod}\gamma \\
% a_1\beta d^{-1}-a_2\alpha d^{-1}-1&\equiv0\mbox{ mod}\gamma
%\end{array}
%\right.
%$$
%because $\nabla^s$ is reflexive, and
$$
\left\{
\begin{array}{cl}
-b_1\beta +b_2\alpha  - d&\equiv0\mbox{ mod}\gamma\\
 a_1\beta -a_2\alpha -d&\equiv0\mbox{ mod}\gamma
\end{array}
\right.
$$
because $\nabla^s$ is reflexive. By solving the system and because $d=\pm 1$, we obtain:
\begin{equation}
\label{eq:abmod}
\left\{
\begin{array}{cl}
\alpha & \equiv a_1 + b_1 + 0 \mbox{ mod}\gamma \\
\beta & \equiv a_2+b_2 + 0 \mbox{ mod}\gamma
\end{array}
\right.
\end{equation}
Let $\langle m_e,v \rangle+1=0$, with $m_e=(A,B)\in M_\varphi$, be the equation in the plane $N_\varphi$ of an edge $e$ of $\nabla^\varphi$ not passing trough $v_z$. The facet of $\nabla^s$ through $v_s$ and $e$ has equation $\langle m_{se},v\rangle+1=0$, where
\begin{equation*}
\label{eq:se}
m_{se} = \left(A, B, -\frac{A\alpha+B\beta+1}{\gamma}\right) \in M.
\end{equation*}
Because $\nabla^s$ is reflexive, we have $A\alpha+B\beta+1 \equiv 0\mbox{ mod}\gamma$. From Eqs.~(\ref{eq:abmod}), it follows that $A(a_1+b_1) + B(a_2+b_2) +1 \equiv 0\mbox{ mod}\gamma$, where $A(a_1+b_1) + B(a_2+b_2) +1\geq 2$ because $e$ does not pass through $v_z$. In particular $\gamma$ is a divisor of an integer $>1$; suppose $\gamma\geq 2$. Given an integer $p$ such that $0<p<\gamma$, we have $\lambda_p:=p\gamma^{-1}\in (0,1)$ and $\gamma\lambda_p\in\mathbb{Z}$. We obtain a contradiction by observing that the point $\lambda_pv_s+(1-\lambda_p)v_z$ in the segment $v_sv_z$ is a lattice point because of Eqs.~(\ref{eq:abmod}).
\end{proof}

\begin{example}\label{89}\cite{Sk99} (A semistable polytope)
Let  $\nabla$ be the reflexive polytope with vertices $v_x=(2,-1,0)$, $v_y=(-1,1,0)$, $v_s=(-1,-1,1)$, $v_t=(-1,-1,-1)$ (polytope 88 in the list by Kreuzer and Skarke  \cite{KrSkweb}). In this case $\nabla^{\varphi}=\nabla^{15;\varphi}$ with vertices $v_x,v_y$ and $v_z=(-1,-1,0)$;  the corresponding toric variety is the weighted projective space $\P^{(2,3,1)}$ with homogeneous coordinates $(x,y,z)$. We take $L$ to be the edge $v_s, v_t$ of lattice length 2. Clearly $\{\nabla^{\varphi},L\}$ generates $N$, and therefore $\nabla$ is a semistable polytope; we will see in Proposition \ref{prop:rif_sum_radd} that this is the only semistable polytope $\nabla^{15,v_z;s}$ with marked section corresponding to $v_z=(-1,-1,0)$. The monomials of the equation of the general $K3$ surface $V$ (in the anticanonical system) are given by
$$
x^3,\,y^2,\,a_{12}z^6, a_8xz^4,\, a_4x^2z^2,\, a_6yz^3,\, a_2xyz,
$$
where $a_i$ is a {general} polynomial in $s, t$ of degree $i$, for $i=2,4,6,8,12$. Applying toric automorphisms we obtain the following equation of $V$
$$
y^2 = x^3+a(s,t)xz^4 + b(s,t)z^6,
$$
where $a,b$ are generic polynomials of degree 8 and 12 respectively. The discriminant $\delta=4a^3+27b^2$ has degree 24 and thus it vanishes in 24 points in $\P^1$. In each of those, the orders of vanishing are $(o(a),o(b),o(\delta)) = (0,0,1)$ and thus there are 24 semistable fibers $I_1$. It is easy to verify that $z=0$ is an irreducible section of the fibration.
\end{example}

\begin{example}\label{1943}(A semistable polytope such that $\{L,\nabla^\varphi\}$  does not generate $N$.)
Let  $\nabla$ be the reflexive polytope with vertices $v_x=(-1,1,0)$, $v_y=(-1,-1,0)$, $v_s=(1,-1,2)$, $v_t=(3,-1,-2)$ (polytope 1943) and $v_z=(2,-1,0) $. We have $\nabla^{\varphi}=\nabla^{15;\varphi}$ with vertices $v_x,v_y$ and $v_z=(2,-1,0)$ (note that, with respect to the polytope in the previous example, we changed the names of the coordinates associated to the vertices of $\nabla^{15;\varphi}$). The edge $L=v_sv_t$ has lattice length 2. The condition is not satisfied: for each point $v\in\nabla^{15;\varphi}$ the matrix $(v,v_z,v_s)$ is not in $SL(3,\Z)$. The monomials in the equation of $V$ are
$$
y^6,\, x^2,\, s^6z^3,\, s^4t^2z^3,\, s^2t^4z^3,\, t^6z^3,\, s^4y^2z^2,\, s^2t^2y^2z^2,\, t^4y^2z^2,\, s^2y^4z,\, t^2y^4z,\, xyz.
$$
Note that the semistable polytope $\nabla^{15,v_z,L';s}$ with $L'$ the segment of vertices $v_t=(2,-1,-1)  $ and $ v_s=(2,-1,1) $ is also reflexive.
\end{example}

%\begin{proposition}\label{15ssz} \todo{TOGLIERE QUESTA PROP. CON DIM. ?} Let  $\nabla^s$ be a semistable polytope generated by the vertices of $\nabla^{15}$  and $v_s, v_t$ such that the
 %edge between $v_s$ and $v_t$ contains $v_z=(-1,-1,0)$. Then $\nabla^s$ is necessary the polytope 89 in Example \ref{89}.
%\end{proposition}
%\begin{proof} \todo{See 10/11/2011 attachment to email. Is this also true for all other toric flexes? Maybe can be done one we get equations of relevant planes}
%\end{proof}

\subsection{Infinity sections, toric flexes}\label{sections}(\emph{On condition 3)})%~\vskip 0.1in

In both  Examples \ref{89} and \ref{1943} $D_z$, namely $z=0$, defines the equation of a  section of the elliptic fibration
 $ V \to \pu$, which is
the restriction of the toric section determined by the divisor $D_z$ in $X$; this section is the same for all the
$K3$ hypersurfaces in the same anticanonical system.
From now on
 we assume that the elliptic fibration has a toric section represented by the divisor $D_z$.

In analogy with the classical Weierstrass model, and following the above notation, let   $\{x,y,z, w_j\}$ be the Cox coordinates corresponding to the vertices of $\nabla^\varphi$ (depending on $\nabla^{\varphi}$ either there is no $v_{w_j}$ or $j\in\{1,2,3\}$), with  $z=0$ be the   defining equation of the toric section $D_z$; let  $s,t, r_k$ be the Cox coordinates corresponding to the remaining vertices of
 $\nabla$, with $v_s$ and $v_t$ be the unit lattice points on the edges through $v_z$, where $\phi(v_s) $ and
 $ \phi(v_t)$ span two different cones of the fan of $\pu$. This will assure the fibration is  easily described in
 terms of the homogeneous
coordinates, see Remark \ref{rem:cyeq}. In many cases this gives a projective resolution of the normal fan.
  %\todo{dobbiamo assumere che il politopo sia simpliciale?  https://filer.case.edu/btn8/homepage/daten/problem.pdf, WHAT IS THE MAXIMAL NUMBER OF VERTICES OF A REFLEXIVE POLYTOPE? BENJAMIN NILL, pagina 1 da' un esempio di politopo riflessivo 3diml non simpliciale, ma non ammette una fibrazione su $\pu$}.
 Let $g(x,y,z, w_j)=0$ be
the equation of
$E$ in $X_{\phi}$ and  $f(x,y,z, w_j, s,t, r_k)=0$ be
the equation of
$V$ in $X$. We often write $f(x,y,z, w_j, s,t, r_k)=G(x,y,z,w_j)$ with $G \in \C[s,t,r_k]$, in the form of the equation of the general elliptic curve in  $ X_\phi$.

\begin{proposition}\label{inftysection}(\emph{Condition 3).})   $f_{|z=0}$ does not depend on the coordinates $(s,t)$ if and only if $v_z$ is an interior
 point
of an edge of $\nabla$ with vertices $v_{s}$ and $v_{t}$.
\end{proposition}
\begin{proof}
The non-zero monomials in the polynomial $f_{|z=0}$  are of the form
\begin{equation}\{s^{\lg m, v_s\rg +1} t^{\lg m, v_t\rg+1}  \prod_{k}r_k^{\lg m, v_{r_k}\rg+1}\} \cdot
x^{\lg m, v_x\rg+1} y^{\lg m, v_y\rg+1}  \prod_{j}w_j^{\lg m, v_{w_j}\rg+1},
 \end{equation} where $m \in M$ satisfies the equation ${\lg m, v_z\rg}+1=0$.
$v_z$  and the vertices $v_{s}$ and $v_{t}$  are collinear if and only if ${\lg m, v_s\rg +1}= \lg m, v_t\rg+1=0$, that is if and only if $f_{|z=0}$.
\end{proof}

\begin{definition} $D_z$ is a \emph{section at infinity} if and only if
$f_{|z=0}$ is independent of the particular point in $\pu$; explicitly there is no
dependence in $(s,t)$. \\
 $f_{|z=0}$ is a
\emph{toric flex} if  the equation
$z=0$ determines one unique point  of the
general elliptic curve $E$ of the fibration.
\end{definition}

If $\nabla^s$ is a semistable polytope, then the toric flex  corresponding to the point $v_z$ as in Proposition
\ref{inftysection} is the analogue of a section at infinity.
We {can} see explicitly the morphisms and the defining equations of the $K3$ surfaces in  Cox coordinates following  \cite{mavlyutovdegenerations} and \cite{brown2010maps}.

\begin{example}
In example \ref{ex:diamond15} the equation of the general $K3$ hypersurface (in the anticanonical system) is:
$$\phi_0x^3+
\phi_1xyz+
\phi_2z^6+
\phi_3y^2+
\phi_4xz^4+
\phi_5x^2z^2+
\phi_6yz^3=0
$$
with each $\phi_j(s,t)$ a generic polynomial of degree $2$ in $(s,t)$.
It is easy to see that $z=0$ is a section, but $f_{|z=0}$ is not a section at infinity.
 The same holds for the other sections coming from toric
divisors  corresponding to $v_{x}=0$ and $v_{y}=0$.
\end{example}

\begin{remark} Under these hypothesis we denote by $s,t$ the two corresponding coordinates.
\end{remark}

We further investigate  the type of sections of the elliptic fibrations in Section \ref{sec:tntsections}.

\subsection{Candelas-Fonts Weierstrass models}\label{subsec:cf}(\emph{On conditions 1), 2) and 3)})%~\vskip 0.1in

Candelas and Font \cite{CandelasFont} fix a two dimensional reflexive polytope ($\nabla^{i;\varphi}$), a toric Fano $B$ (in particular $B=\P^1$) and proceed, by what it is called ``un-Higgsing" in the physics literature, to give examples of elliptic fibrations with interesting singular fibers (\emph{``gauge groups"}) of Calabi-Yau varieties in toric Fano, fibered over $B$ where the general elliptic curve of the fibration is the general elliptic curve in the anticanonical system of the fan over the fixed polytope $\nabla^{i;\varphi}$.
We recast their construction in terms of elliptic fibrations with section at infinity and the semistable polytope.

 \begin{definition}\label{WModel} A \emph{ Candelas-Font Weierstrass  model}   $W \to \pu$  is an elliptically fibered  $K3$  with orbifold Gorenstein singularities, not necessarily general in the anticanonical system, in a %semistable
 variety $X^{i,v_z,L;s}$ with general fiber $E \subset X_{i;\phi}$ %$\nabla^{i;\varphi}$
 and a section at infinity in $D'_z$ (the toric divisor corresponding to $v_z$).
\end{definition}
Note that these singularities are canonical; general anticanonical hypersurfaces in the projective resolution of Fano varieties have orbifold Gorestein singularities.

Next we prove a sufficient condition for an elliptically fibered general  $K3$ with general fiber $E \subset X_{\phi}$ %$\nabla^\varphi$
to have a Candelas-Font Weierstrass model and we express the condition in term of the combinatorics of the polytope as well as the geometry.

\begin{theorem}\label{th:diagr}(\emph{Candelas-Font Weierstrass models ({Conditions 1) and 3)}).}) Let  $\nabla^{i;\varphi}\subset N_{\varphi,\R}$ be a 2-dimensional polytope and $v_z$ a vertex of $\nabla^{i;\varphi}$ such that there exists a semistable polytope $\nabla^{i,v_z,L;s}\subset N_{\R}$.
A general elliptically fibered $K3$  hypersurface $V $ in the anti-canonical of a toric Fano threefold $\X$ with section at infinity in the edge $L$ and general fiber $ E \subset X_{i;\phi}$ is birationally equivalent to a Candelas-Font Weierstrass model $W \subset X^{i,v_z,L;s}$.
\end{theorem}
\begin{proof} By hypothesis the polytope $\nabla$  over the fan of $\X$ contains the polytope $\nabla^{i;s}:=\nabla^{i,v_z,L;s}$, hence we have the dual inclusion
$\Delta \subset \Delta^{i;s}$. $\Delta$ defines a linear subsystem $\mathcal L  \ \subset  \ |-K_{X^{i;s}}|$; the resolution of the interminancy locus provides a birational morphism $\P_{\D} \to X^{i;s} $.  $\bar V$, the general hypersurface in  $\mathcal L $, is the strict transform of the general hypersurface $W$ in $|-K_{X^{i;s}}|$ . This induces a birational morphism between the pullback projective resolution, that is $V \to W$.
\end{proof}

\begin{theorem}\label{th:projection}(\emph{Conditions 1), 2) and 3)}) Assume that a Newton (moment) polytope is a reflexive subpolytope $\Delta \subset \Delta^{i;s}$  which contains $\Delta^{i;\varphi}$ the dual of $\nabla^{i;\varphi} \subset \nabla^{i;s}$. Then also the  viceversa of Theorem \ref{th:diagr} holds, that is  the dual of $\Delta$, $\nabla$  is contained in $\nabla^{X^{i;s}}$.

Furthermore the  projective resolution  of the corresponding $K3$  has a section at infinity.

The  condition that $\nabla$ projects onto $\nabla^{i;\varphi}$ together with the existence of a section at infinity characterizes the examples in Candelas' algorithm.
\end{theorem}
\begin{proof} It is enough to observe that $\nabla ^{i;s} \subset \nabla$ and that $\nabla$ projects onto $\nabla^{i;\varphi}$.
\end{proof}

The transformations can be written explicitly in Cox coordinates, as in \cite{mavlyutovdegenerations} and \cite{brown2010maps}:
\begin{example}
Example \ref{3737}: we obtain the projective subdivision of the normal fan to $\D$ obtained adding the rays $v_x,v_y,v_z$.
The monomials defining the generic $K3$ in the corresponding toric variety $X$ are
%$$ x^3a^3b^3,\, xyzb^2d^2t^2,\, xyza^2c^2s^2,\, xyzabcdst,\,z^6s^6t^6,\,y^2c^2d^2,\,xz^4abs^4t^4,\,x^2z^2a^2b^2s^2t^2,\,yz^3cds^3t^3. $$
$$ x^3a^3b^3,\, xyzb^2d^2t^2,\, xyza^2c^2s^2,\, xyzabcdst,\,z^6s^6t^6,\,y^2c^2d^2,\,xz^4abs^4t^4,$$
$$
x^2z^2a^2b^2s^2t^2,\,yz^3cds^3t^3.
$$
Let $L$ be the edge between the vertices $v_a$ and $v_b$ and consider the semistable polytope $\nabla^{15,v_x,L;s}$. The polytopes satisfy the hypothesis of Theorem \ref{th:diagr} and Theorem \ref{th:projection};
the equation of the Candelas Weierstrass model $W$ has monomials
$$
x^3,\,y^2,\, a_{12}z^6, a_8xz^4,\, a_4x^2z^2,\, a_6yz^3,\, a_2xyz,
$$
where $a_i$ is a general polynomial in $s, t$ of degree $i$, for $i=2,4,6,8,12$.

Examples \ref{4318}  and \ref{1943}  satisfy the hypothesis of Theorem \ref{th:diagr} and it is likewise possible to write the transformation to their respective  Candeals-Font Weierstrass models.

It is then easy to write the equations of the discriminant of the elliptic fibrations.
\end{example}

Theorem \ref{th:diagr} provides a criterion for constructing examples $E_8 \times E_8$ F-theory Heterotic duality in the toric context.
In can be verified that all the ``gauge groups", that is all the singular fibers of the elliptic fibrations, are  subgroups of $E_8 \times E_8$.

Braun in \cite{Braun} builds toric Weierstrass models  by contracting  some toric divisors associated to
$\nabla^\varphi$, and thus changing the elliptic fiber; in view of Section \ref{symplecticcut} we instead keep the
basis fixed.  Also \cite{Braun} considers elliptic Calabi-Yau threefolds in certain points of the moduli, while we
consider an appropriate embedding of the elliptic fibration which makes it general as hypersurface in anticanonical system of the toric ambient
 space: this makes the resolution process straightforward as it is induced by the resolution of the toric ambient space.
  It would be interesting to combine the methods.

\section{Toric and non toric sections}\label{sec:tntsections}

We discuss combinatorial conditions for the existence of sections of the elliptic fibration. Let $\nabla\subset N_{\R}$ be a 3-dimensional reflexive polytope containing a 2-dimensional reflexive polytope $\nabla^{\varphi}$ and $v_z$ be a vertex of $\nabla^{\varphi}$. The equation of $N_{\varphi,\R}$ in $N_{\R}$ is $\langle m_{\varphi}, v\rangle =0$, without loss of generality we take $m_{\varphi}=(0,0,1)$. Consider the {maximal} projective subdivision $\Sigma_{\max}$ of the fan over the proper faces of $\nabla$ and let $V_{\max}$ be the smooth $K3$ hypersurface in the corresponding smooth Fano toric variety $X_{\max}$. Let $v_s$ be a lattice point in $\nabla$ at lattice distance one from $v_z$ along an edge of $\nabla$ through $v_z$ not in $N_{\varphi, \R}$. At last, let $D'_z, D'_s$ be the intersection of $D_z,D_s \subset X_{\max}$ with $V_{\max}$. It can be shown that the fiber of the elliptic fibration $V_{\max}\rightarrow \P^1$ is linearly equivalent to $\sum_{v_i\in \nabla_{top}} \langle m_{\varphi},v_i\rangle D'_i$, where $\nabla_{top}=\{v\in \nabla| \langle m_{\varphi},v\rangle >0 \}$  \cite{PeSk97}.

\begin{theorem}
\label{prop:criterio_sections}
Let $v_1,v_2$ be the two lattice neighbors of $v_z$ along the two edges of $\nabla^\varphi$ through $v_z$. Suppose $v_zv_1v_s$ and $v_zv_2v_s$ are elementary triangles  in the maximal triangulation corresponding to $\Sigma_{\max}$, $D'_z,D'_s$ are irreducible and $\nabla$ is simple. If $v_z=v_1+v_2$ then $D'_z$ is a (toric) section of the elliptic fibration; moreover the converse is also true.
\end{theorem}

The hypothesis on $D'_s$ is what happens in the cases considered in the previous sections. In fact:
\begin{remark} If $v_s$ is a vertex of the polytope $\nabla$, then $D'_s$ is irreducible. This is the case of the basic semistable models. If $v_s$ is not a vertex of $\nabla$ and
 $v_z$ is on the interior of the same edge, then $D'_z$ is irreducible if and only if $D'_s$ is.
\end{remark}
See also Proposition
\ref{prop:rif_sum_radd}.
\begin{remark}
The theorem has the hypothesis that $\nabla$ is simple, in particular it is sufficient to ask $\nabla$ simple at $v_z$. All the reflexive polytopes (which induce elliptic fibrations) we examined in order to prepare this paper satisfy this condition.
\end{remark}

\begin{proof}
We take $v_1=(a_1,a_2,0), v_2=(b_1,b_2,0), v_z=(z_1,z_2,0), v_s=(\alpha,\beta,\gamma)\in \nabla_{top}$ (i.e. $\gamma>0$). Moreover $v_z=\lambda v_1+\mu v_2$, for $\lambda, \mu\neq 0$.

$D'_z$ is section iff $D'_z\cdot \sum_{v_i\in \nabla_{top}} \langle m_{\varphi},v_i\rangle D'_i =1$. Because $\nabla$ is simple at $v_z$, the only lattice point in $\nabla$ at lattice distance one along an edge of $\nabla_{top}$ through $v_z$ is $v_s$. We can distribute the intersection over the sum and observe that by the discussion in Section ~\ref{subsec:intersection} all the summands but $\langle m_{\varphi}, v_s\rangle D'_z\cdot D'_s = \gamma D'_z\cdot D'_s$ are null. Therefore $D'_z$ is a section iff $\gamma D'_z\cdot D'_s=1$.

Because $v_zv_1v_s$, $v_zv_2v_s$ are elementary triangles, it is straightforward to verify that $\gamma \lambda d =\pm 1$ and $\gamma \mu d = \pm 1$, where $d = a_1b_2-a_2b_1$. Moreover, by Theorem~\ref{th:intK31}:
$$
D'_z\cdot D'_s=\lg m_{zs1},v_2\rg+1=\lg m_{zs2},v_1\rg +1,
$$
where $m_{zs1}, m_{zs2}\in M$ are the dual points of the facets of $\nabla$ span by $v_z,v_1,v_s$ and $v_z,v_2,v_s$ respectively. %We need to prove that $\lambda=\mu=1$ $\Leftrightarrow$ $\gamma=1$ and $\langle m_{zs1},v_2\rangle=\langle m_{zs2},v_1\rangle =0$.
$$
\langle m_{zs1},v_2\rangle \propto
 \left |
\begin{array}{ccc}
b_1 & \alpha-z_1 & a_1-z_1 \\
b_2 & \beta-z_2 & a_2-z_2 \\
0 & \gamma & 0
\end{array}
\right |
=
\left |
\begin{array}{ccc}
b_1 & \alpha-\lambda a_1 - \mu b_1 & (1-\lambda)a_1-\mu b_1 \\
b_2 & \beta-\lambda a_2 - \mu b_2 & (1-\lambda)a_2-\mu b_2 \\
0 & \gamma & 0
\end{array}
\right |
$$
Therefore $\langle m_{zs1},v_2\rangle = 0$ iff
$$
\left|
\begin{array}{cc}
b_1 & (1-\lambda) a_1 \\
b_2 & (1-\lambda) a_2
\end{array}
\right|
=0,
$$
and since $v_1$ and $v_2$ are linearly independent this is the case if and only if $\lambda=1$. A similar argument shows that $\langle m_{zs2},v_1\rangle =0$ iff $\mu=1$.

On one hand, if $v_z=v_1+v_2$ then we have $d:= a_1b_2-a_2b_1=\pm 1$ (as we already observed in the proof of Proposition~\ref{prop:rif_sum_radd}) and $\lambda=\mu=1$. Therefore $\langle m_{zs1},v_2\rangle = \langle m_{zs2},v_1\rangle =0$, $\gamma=1$ and therefore $D'_z$ is a section. On the other hand, if $D'_z$ is a section, then $\lambda=\mu=\gamma=1$ and in particular $v_z=v_1+v_2$.

\end{proof}

 \section{Symplectic cut, degenerations, physics duality}\label{symplecticcut}

  In this section we assume that the polytope $\Delta$ is simple, that is, the normal toric variety $\P_\Delta$ is simplicial \cite{CoxLSh}; this is the case in all the Candelas' examples
   %We show that a reflexive polytopes $\nabla$ which projects onto $\nabla^i$ and has a section at infinity satisfies the hypothesis of Hu's result; the weakly semistable degeneration  induces a semistable degeneration of the general elliptically fibered K3 hypersurface to two rational elliptic surfaces which intersect along a smooth elliptic fiber.

%\subsection{F-theory models; F-theory/Heterotic duality}
%\begin{ptheorem}[Candelas]Given $F[X, \pi_j, \sigma_j]$:
%\begin{description}
%    \item[a] The polytope of Heterotic dual variety $\bna _{Y}$ is obtained by a   suitable projection of the polytope $\bna_{X} $
  %  \item[b] The groups $(G_1, G_2)$ are read from the polytope $\bna_{X}  $.
%\end{description}
%\end{ptheorem}
\begin{lemma}\label{topbottom}
 If $\Delta \subset M_\R$ is a  polytope associated to a toric Fano threefold $\P_{\D}$ satisfying the condition 1) of Definition \ref{sec:candelasexamples}, then the following are equivalent:
\begin{enumerate}[i.]
\item 2) holds,
 \item  If $D$ is a facet, a codimension $1$ face in $\nabla$, with inner normal vector $\nu_D= (w_1, \dots, w_n)$, then: $w_n >0$ (resp. $w_n <0$) if and only if $D$ lies entirely in the half space
$N_{\leq 0} =\{(z _1, \dots, z_n) \in N \text{ such that } z_n \leq 0\}$ (resp $N_{\geq 0}$).
 \end{enumerate}
\end{lemma}
\begin{proof}
 \emph{i.} $\Longleftrightarrow $ \emph{ii.} : If $n=2$, then the statement is immediate, as $\nabla$ is convex.
Otherwise, let us consider the plane passing through the $z_n$ axis and parallel to $\nu_D$, and  let $D_2$ be the intersection of $D$ with such a plane. We can then reduce to the case $n=2$.
\end{proof}

 In \cite{Hu06}, S. Hu shows that suitable partitions of \emph{simple} polytopes $\Delta \subset M_\R$ induces a  semistable (or weakly semistable) degeneration of the toric variety associated to the polytope.

\begin{theorem}\label{semistablepartition} Let $\Delta \subset M_\R$ be the polytope associated to a toric Fano threefold $X$ satisfying conditions 1) and 2) of Definition \ref{sec:candelasexamples}. We also assume that $\Delta$ is simple, that is $X$ is simplicial. Then the polytope $\Delta^{\varphi} \subset \Delta$, dual of $\nabla^{\varphi}$, determines a  \emph{symplectic cut}, a simple, semistable partition of $\Delta$ \cite{Hu06}.
\end{theorem}
\begin{proof}  $\Delta^{\varphi}$ divides the polytope $\Delta$ in two polytopes $\Delta_1$ and $\Delta_2$. Lemma \ref{topbottom} shows that each $\Delta_j, \ j=1,2$, is simple, that is the partition is simple.  We need to verify that the conditions stated in \cite{Hu06} are satisfied, namely that  any $\ell$-face  of $\Delta_j$, $\ell=1,2$, is contained in exactly $k - \ell + 1$ polytopes $\Delta_j$ if there is a $k$-face of $\Delta$ containing it.
 This follows from a straightforward verification.
\end{proof}

\begin{lemma} Let $\Delta$ be a polytope as in \ref{semistablepartition}. Then the semistable partition determined by $\Delta^{\varphi}$ is also \emph{balanced}, in the sense of \cite{Hu06}.
\end{lemma}
\begin{proof} By construction all the vertices of $\Delta_j$ which are not vertices of $\Delta$ lie on an edge of $\Delta$, which makes the subdivision balanced. Note that these vertices are the vertices of $\Delta^{\varphi}$ which are not vertices of $\Delta$.
\end{proof}

\begin{definition} \cite{Hu06} The semistable, balanced subdivision of $\Delta$ determined by $\Delta^{\varphi}$ is  mildly singular if the vertices of $\Delta_j$  which are not vertices of $\Delta$ are non singular in each $\Delta_j $, that is the primitive vectors at such vertex span the lattice $M$ (over $\Z$).
\end{definition}

\begin{theorem}\label{mah}(Th. 3.5 \cite{Hu06}) Let $\{ \D_j \}$, $j=1,2$ be a mildly singular semistable partition of
$\D$: then there exists a weak semistable degeneration of $\P_{\D} $, $f: {\tilde
\P_{\D}} \to \C$ with central fiber ${\P_{\D}}_{,0}=\cup _j {\P_{\D_j}}$. The central
fiber is completely described by the polytope partition $\{ \D_j\}$ and $\P_{\D_1} \cap \P_{\D_2}= \P_{\Delta^{\varphi}}$.
\end{theorem}
The following corollary follows from Lemma \ref{topbottom}:

\begin{theorem}\label{cf} Let $\Delta \subset M_{\R}$ be the polytope associated to a toric Fano threefold $X$ satisfying conditions 1) and 2) of Definition \ref{sec:candelasexamples}. We also assume that $\Delta$ is simple and that $\Delta^{\varphi}\subset \D$, the dual of $\nabla^{\varphi}$, determines a  simple, mildly singular semistable partition of $\Delta$ \cite{Hu06}. Let $f: {\tilde
\P_{\D}} \to \C$ be the induced weak semistable degeneration.  The rays of the toric fan of $\P_{\D_1}$ are the rays of $\P_\D$ with $z_3 \geq 0$
    together with the ray $z_1=z_2=0, \ z_3 \geq 0$;
    similarly for the rays of the toric fan of $\P_{\D_2}$ (with  the ray $z_1=z_2=0, \ z_3 \leq 0$).
\end{theorem}

The degeneration of Theorem \ref{mah} induces a degeneration of the general hypersurface $\bar V \subset \P_\Delta $, Section 4 \cite{Hu06}; let $\mathcal L_j$ be the (ample) line bundle on $\P_{\D_j}$ associated to $\Delta_j$ and $\bar S_1$ and $\bar S_2$ be the associated hypersurfaces.
If all the vertices of $\nabla^{\varphi}$ are also  vertices of $\nabla$, then $\P_{\D}$, $\P_{\D_1}$ and $\P_{\D_2}$  have a fibration with general fiber $\P_{\Delta^{\varphi}}$ and the degeneration  preserves the fibration. We have proved the following:
 \begin{theorem}\label{toricbat} Theorem \ref{mah} induces a weakly semistable degeneration of the general hypersurface $\bar V$ to $\bar S_1  \cup \bar S_2$; in addition $\bar S_1 \cap \bar S_2$ is the general elliptic curve  in $\Delta^{\varphi}$. The construction of the degeneration shows that there is a naturally induced semistable degeneration of the maximal resolution $V$ to $S_1$ and $S_2$.
 \end{theorem}

In addition:

\begin{proposition}
 The surface $\bar S_j \in |{\mathcal L}_{Z_j}|$ is a rational elliptic surface, where ${\mathcal L}_{Z_j}$
  is the line bundle determined by the polytope $\D_{Z_j}$.
\end{proposition}
\begin{proof}
Note in fact that $-K_{Z_1}=  \sum_{v_k \in \Sigma^{ (1)}} v_k$ and
that ${\mathcal L}_{Z_1} = + \sum_{v_k \in \Sigma^{(1)}} v_k - e_3$.
Hence $K_{S_1}= -{e_3}_{|{S_1}}$.
 Note also that $e_3=-K_{S_1}$ is the divisor of a general fiber of
 $\pi_1: Z_1 \to \pu$. \end{proof}

If a section at infinity exists $S_j \to  \bar S_j$ is not a crepant resolution; the exceptional curve is a section of the elliptic fibration.

\begin{remark}\label{rmk:cutswithinftysection}Condition 3) assures that  all the K3 hypersurfaces in the anticanonical of the toric Fano have the same infinity in section:
this is consistent with the F-theory/Heterotic duality, which  is for families of K3.
\end{remark}

%\begin{corollary} What if the stick is only on the top or the bottom? Need to look better at the cones in the lift of $\bna_Z$. %\end{corollary}

\begin{example}  Examples \ref{3737}, \ref{ex:diamond15}, \ref{89}, \ref{example:diamond1} satisfy the hypothesis of the above Theorems.
Note that the general $K3$ hypersurface in the ``diamond" \ref{example:diamond1} $X \to \pu$ with  fiber $ E \subset \P^2$ does not have a section; but it has a semistable degeneration induced by the symplectic cut. In this case all the vertices of $\nabla^{\varphi}$ are also vertices of $\nabla$.
\end{example}

Note that the examples \ref{example:diamond1} and \ref{ex:diamond15}, which admit a semistable degeneration as hypersurfaces in toric Fano, do not satisfy condition 3) in the Candelas' conditions.
 In fact, the symplectic cut construction is more general, and holds for more examples,
with no condition on the existence of a section.

\begin{example}
Example \ref{4318} does not satisfy condition 2) and it does not have a semistable degeneration in the sense of the above theorems.
\end{example}

\bigskip

\begin{appendix}

%<insert any appendices here>

\section{Equations of elliptic curves in toric Del Pezzo surfaces}
\label{delpezzo}

\begin{figure}[htbp]
\begin{center}
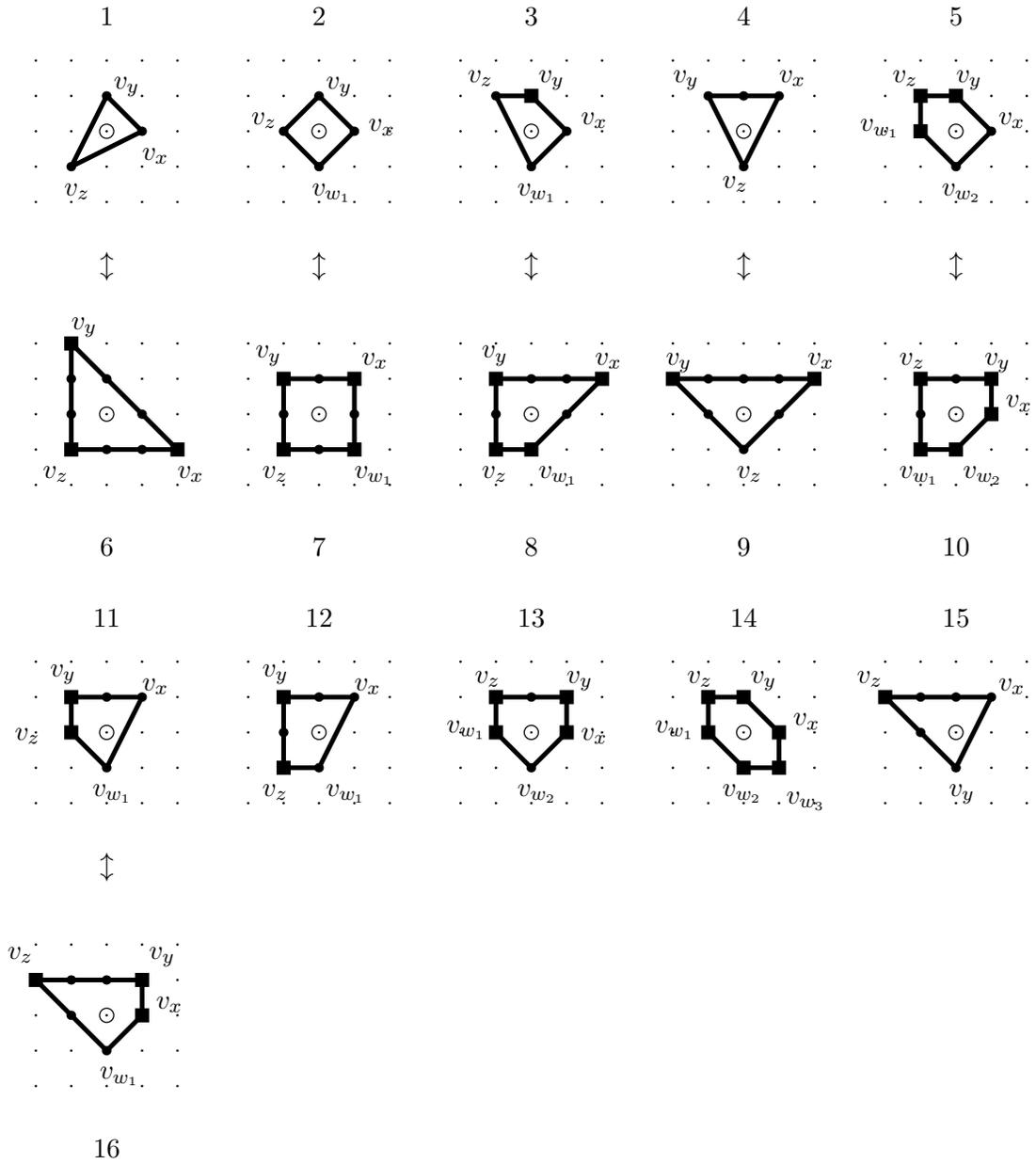
\caption{Reflexive polytopes in the plane}
\label{fig:sections2d}
\end{center}
\end{figure}

Figure~\ref{fig:sections2d} depicts all 2-dimensional reflexive polytopes $\nabla^{i;\varphi}$ up to $\SL(2,\Z)$ transformations. In the figure, an arrow denotes a pair of dual polytopes, if there is no arrow, the polytope is autodual. Let $\Sigma_{i;\varphi}$ be the normal fan of $\D^{i;\varphi}$ and $X_{i;\phi}$ the corresponding (Del Pezzo) toric variety: $X_{i;\phi}=(\C^r-Z(\Sigma_{i;\varphi}))/G_{i;\varphi}$, where $r$ is the number of vertices of $\nabla^{i;\varphi}$. For each $i$ we compute the equation of the generic elliptic curve embedded in $X_{i;\phi}$. We name the vertices as in the figure.

\begin{itemize}

\item $i=1$ (i.e. here we consider $\nabla^{1;\varphi}$ and its dual $\D^{d(1),\varphi}=\D^{6,\varphi}$): $X_{1;\phi}=\P^2_{(x,y,z)}$; the monomials in the equation of the generic elliptic curve are
$$
x^3,y^3,z^3,yz^2,y^2z,xz^2,xyz,xy^2,x^2z,x^2y.
$$

\item $i=6$ (i.e. here we consider $\nabla^{6;\varphi}$ and its dual $\D^{d(6);\varphi} = \D^{1;\varphi}$):  $X_{6;\phi}=\P^{2}_{(x,y,z)}/\Z_3$, where the relations are
$$
(x,y,z)=(\mu,\epsilon_k)\cdot(x,y,z)=\left(\mu x, \epsilon_k \mu y, \frac{1}{\epsilon_k}\mu z\right),
$$
with $\mu\in\C^*$ and $\epsilon_k=e^{2\pi i \frac{k}{3}},\, k=0,1,2$. The monomials in the equation of the generic elliptic curve are
$$
x^3,y^3,z^3,xyz.
$$

\item $i=2:$ $X_{2;\phi}=\P^1\times\P^1$, the monomials in the equation of the generic elliptic curve are
$$
y^2z^2,x^2y^2,z^2w_1^2,x^2w_1^2,yz^2w_1,xzw_1^2,xyzw_1,xy^2z,x^2yw_1.
$$

\item $i=7:$ $X_{7;\phi}=(\P^1_{(x, z)}\times\P^1_{(y, w_1)})/\Z_2$, where the relations are
$$
(x,y,z,w_1)=(\mu,\lambda,\bar{0})\cdot (x,y,z,w_1)=(\mu x, \lambda y, \mu z, \lambda w_1),
$$
and
$$
(x,y,z,w_1)=(\mu,\lambda,\bar{1})\cdot (x,y,z,w_1)=(\mu x, \lambda y, -\mu z, -\lambda w_1),
$$
Monomials:
$$
x^2w_1^2,x^2y^2,z^2w_1^2,xyzw_1.
$$

\item $i=3$: $X_{3;\phi}=\C^4-Z(\Sigma_{3;\varphi})/
G_{3;\varphi}$, where $Z(\Sigma_{3;\varphi})=V(x,z)\cup V(y,w_1)$ and the relations are
$$
(x,y,z,w_1)=(\mu x, \lambda y, \mu z, \mu\lambda w_1),\, \mu,\lambda\in\C^*.
$$
Monomials:
$$
y^2z^3,x^3y^2,zw_1^2,xw_1^2,yz^2w_1,xyzw_1,x^2yw_1,xy^2z^2,x^2y^2z.
$$

\item $i=8$: Monomials:
$$
x^2y^2,x^3w_1,y^3z,z^2w_1^2,xyzw_1.
$$

\item $i=4:$ $X_{4;\phi}=\P^{(1,1,2)}_{(x,y,z)}$. Monomials:
$$
y^4,x^4,z^2,x^3y,x^2y^2,xy^3,x^2z,xyz,y^2z.
$$

%Nella Tesi ho preso $v_x=(0,-1),v_y=(1,1),v_z=(-1,1)$, ottenendo il piano proiettivo pesato di pesi (2,1,1). I monomi della generica curva ellittica sono
%$$
%y^4,x^2,z^4,yz^3,y^2z^2,y^3z,xz^2,xyz,xy^2.
%$$

\item $i=9:$ $X_{9;\phi}=\P^{(1,1,2)}_{(x,y,z)}/\Z_2$, where the relations are:
$$
(x,y,z)=(\mu,\bar{0})\cdot (x,y,z)= (\mu x, \mu y, \mu^2 z),
$$
and
$$
(x,y,z)=(\mu,\bar{1})\cdot [x,y,z]= (-\mu x, \mu y,- \mu^2 z),\, \mu\in \C^*.
$$
Monomials:
$$
y^4,x^4,z^2,x^2y^2,xyz.
$$

\item $i=5$:
$$
y^2z^3w_1^2,x^2y^2z,zw_1^2w_2^2,xw_1w_2^2,x^2yw_2,yz^2w_1^2w_2,xyzw_1w_2,xy^2z^2w_1.
$$

\item $i=10:$
$$
x^2y^2w_2,xy^2z^2,z^2w_1^2w_2,xw_1^2w_2^2,yz^3w_1,xyzw_1w_2.
$$

\item $i=11:$
$$
y^4z^3,x^3y,zw_1^2,x^2w_1,xy^3z^2,x^2y^2z,y^2z^2w_1,xyzw_1.
$$

\item $i=16:$
$$
yz^4,x^2y^3,z^3w_1,xw_1^2,xy^2z^2,xyzw_1.
$$

\item $i=12:$
$$
x^2y^2,x^2w_1,y^4z^2,z^2w_1^2,y^2z^2w_1,xy^3z,xyzw_1.
$$

\item $i=13:$
$$
yz^3w_1^2,x^2y^3z,z^2w_1^2w_2,xw_1w_2^2,x^2y^2w_2,xy^2z^2w_1,xyzw_1w_2.
$$

\item $i=14:$
$$
xy^2z^2w_1, x^2y^2zw_3, zw_1^2w_2^2w_3,xw_1w_2^2w_3^2,yz^2w_1^2w_2,x^2yw_2w_3^2,xyzw_1w_2w_3.
$$

\item $i=15:$ $X_{15;\phi}=\P^{(2,3,1)}_{(x,y,z)}$. Monomials:
$$
z^6,x^3,y^2,xz^4,x^2z^2,yz^3,xyz.
$$
\end{itemize}

\section{Ellitpic fibration in homogeneous coordinates}

\begin{example}
Consider the reflexive polytope $\nabla\subset N$ is the 3-dimensional reflexive polytope with vertices $v_x=(1,0,0), v_y=(0,1,0), v_s=(-1,-1,1), v_t=(-1,-1,-1)$ (polytope number 1 in the list \cite{KrSkweb}). $\nabla^{\varphi}$ is the 2-dimensional subpolytope of $\nabla$ given by the vertices $v_x, v_y, v_z=(-1,-1,0)$, $\nabla^{\varphi}=\nabla^{1,\varphi}$. The dual $\D\subset M_{\R}$ to is the reflexive polytope with vertices $(2,-1,0)$, $(-1,2,0)$, $(-1,-1,3)$, $(-1,-1,-3)$, see Figures~\ref{fig:1withstick}, \ref{fig:6withstick}.

\begin{figure}[t]
\begin{center}
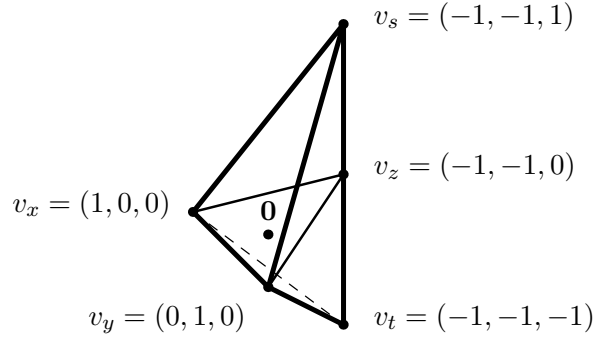
\caption{The polytope $\nabla \subset N_{\R}$}
\label{fig:1withstick}
\end{center}
\end{figure}

\begin{figure}[t]
\begin{center}
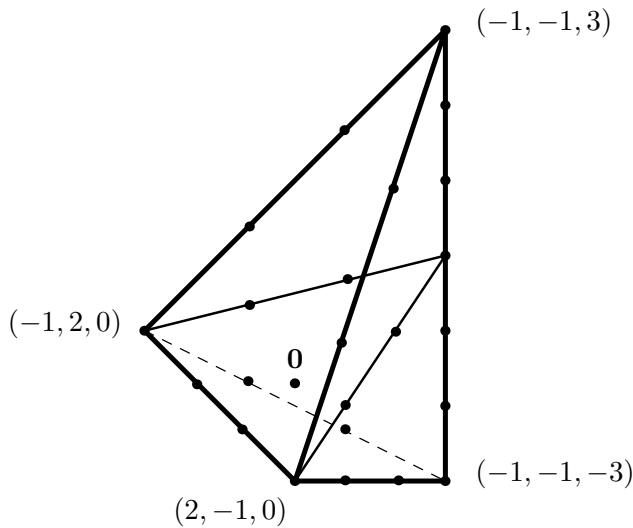
\caption{The polytope $\D\subset M_{\R}$ dual to $\nabla \subset N_{\R}$}
\label{fig:6withstick}
\end{center}
\end{figure}

Consider the fan $\Sigma$ with rays $v_x,v_y,v_z, v_s, v_t$. We have $X_{\phi}=\P^2_{(x,y,z)}$, and $X:=X_{\Sigma}=(\C^3-Z(\Sigma))/(\C^*)^2$, where
\begin{equation}
\label{eq:exfib}
(x,y,z,s,t)\sim(\lambda x,\lambda y,\lambda \mu^2 z, \mu^{-1}s,\mu^{-1}t),\, \lambda\cdot\mu\neq 0.
\end{equation}

Observe that $\nabla^{\varphi}$ lies on the lattice $N_{\varphi}=\{v\in N: \langle v,m_{\varphi}\rangle=0\}$, where $m_{\varphi}=(0,0,1)\in M$. Moreover
$$
\nabla=\nabla^{\varphi}\cup \nabla_{top}\cup \nabla_{bottom},
$$
where $\nabla_{top}=\{v\in \nabla| \langle v,m_{\varphi}\rangle>0\}$ and $\nabla_{top}=\{v\in \nabla |\langle v,m_{\varphi}\rangle<0\}$.

The homogeneous coordinates for the base of the fibration $\P^1$ are given by $(z_{top},z_{bottom})$ with
$$
z_{top}=\prod_{v_i \in \nabla_{top}}z_i^{\lg v_i,m_{\varphi}\rg}
$$
and
$$
z_{bottom}=\prod_{v_i \in \nabla_{bottom}}z_i^{-\lg v_i,m_{\varphi}\rg}.
$$
In this case we have $z_{top}=s$ and $z_{bottom}=t$. It is clear that if we fix a point $(s,t)$ with $s,t\neq 0$, we obtain as a fiber a whole copy of $\P^2$. This can also be seen using the equivalence relation (\ref{eq:exfib}). The generic $K3$ hypersurface $V$ in $X$ has equation
%$$
%a_1x^3+x^2(a_2z+p_1^{(2)}y)+x(p_1^{(4)}z^2+a_3y^2+p_2^{(2)}yz)+a_4y^3+p_3^{(2)}y^2z+p_2^{(4)}yz^2+p_1^{(6)}z^3=0,
%$$
\begin{multline*}
a_1x^3+x^2(a_2z+p_1^{(2)}y)+x(p_1^{(4)}z^2+a_3y^2+p_2^{(2)}yz)\\+a_4y^3+p_3^{(2)}y^2z+p_2^{(4)}yz^2+p_1^{(6)}z^3=0,
\end{multline*}
where the $a_i$ are generic complex numbers, and the $p_i^{(j)}$ are generic homogeneous polynomials of degree $j=2,4,6$ in $s,t$. Fixing a point in the base space amounts to fixing the values of the $p_i^{(j)}$. The generic fiber of the fibration restricted to $V$ is a smooth cubic curve (i.e. an elliptic curve) in $\P^2_{(x,y,z)}$.

\end{example}

\end{appendix}

\bibliographystyle{my-h-elsevier}

\end{document}